\documentclass[11pt]{MathFormatv1}

\def\block(#1,#2)#3{\multicolumn{#2}{c}{\multirow{#1}{*}{$ #3 $}}}

\def\C{\mathcal{C}}
\def\D{\mathcal{D}}

\def\Z{\mathbb{Z}}
\def\Q{\mathbb{Q}}
\def\F{\mathbb{F}}

\def\N{\mathcal{N}}

\def\End{\mathrm{End}}
\def\Hom{\mathrm{Hom}}
\def\Ker{\mathrm{Ker}}

\def\O{\mathcal{O}}

\def\ind{\mathop{\mathrm{ind}}\nolimits}

\def\Vec{\mathrm{Vec}}
\def\sVec{\mathrm{sVec}}
\def\Ver{\mathrm{Ver}}
\def\k{\mathbf{k}}
\def\ev{\mathrm{ev}}
\def\coev{\mathrm{coev}}
\def\Rep{\mathop{\mathrm{Rep}}\nolimits}
\def\inv{\mathrm{inv}}
\def\coker{\mathop{\mathrm{coker}}\nolimits}
\def\length{\mathop{\mathrm{length}}\nolimits}
\def\gr{\mathop{\mathrm{gr}}\nolimits}

\begin{document}

\title{Hilbert Basis Theorem and Finite Generation of Invariants in Symmetric Fusion Categories in Positive Characteristic}

\author{Siddharth Venkatesh}

\begin{abstract} 

In this paper, we conjecture an extension of the Hilbert basis theorem and the finite generation of invariants to commutative algebras in symmetric finite tensor categories over fields of positive characteristic. We prove the conjecture in the case of semisimple categories and more generally in the case of categories with fiber functors to the characteristic $p > 0$ Verlinde category of $SL_{2}$.  We also construct a symmetric finite tensor category $\sVec_{2}$ over fields of characteristic $2$ and show that it is a candidate for the category of supervector spaces in this characteristic. We further show that $\sVec_{2}$ does not fiber over the characteristic $2$ Verlinde category of $SL_{2}$ and then prove the conjecture for any category fibered over $\sVec_{2}$.

\end{abstract}

\maketitle

\section{Introduction}

\subsection{} Let $\k$ be an algebraically closed field of characteristic $p \ge 0$. We begin by recalling some basic definitions. 

Recall that a \textit{monoidal category} is a category $\C$ equipped with a tensor product bifunctor $\otimes$ and a unit object $\mathbf{1}$ that satisfy certain associativity and unit axioms (see \cite[2.1]{EGNO} for more details.) A monoidal category is called \textit{rigid} if left and right duals exist for every object in $\C$ (see \cite[2.10]{EGNO}) and is called \textit{braided} if there exists a natural commutativity isomorphism 

$$c_{X, Y} : X \otimes Y \rightarrow Y \otimes X$$
satisfying certain compatibility conditions (see \cite[8.1]{EGNO}). If $c$ also satisfies 

$$c_{X, Y} \circ c_{Y, X} = \id_{Y \otimes X}$$
for each $X, Y \in \C$, then we say that the monoidal category $\C$ is \textit{symmetric} (\cite[9.9]{EGNO}).

Recall that a $\k$-linear abelian category $\C$ is said to be $\textit{locally finite}$ (\cite[1.8]{EGNO}) if the following two conditions are satisfied:

\begin{itemize}
\item[1.] For any two objects $X, Y \in \C$, the $\k$-vector space $\Hom_{\C}(X, Y)$ is finite dimensional.

\item[2.] Every object in $\C$ has finite length.

\end{itemize}
$\C$ is called \textit{finite} if, in addition to the above two properties, the following two conditions also hold:

\begin{enumerate}

\item[1.] $\C$ has enough projectives, i.e., every simple object in $\C$ has a projective cover.

\item[2.] There are finitely many isomorphism classes of simple objects in $\C$.

\end{enumerate}

\vspace{0.2cm}

\begin{definition} \label{tensor}

A \textit{tensor category} (see \cite[4.1]{EGNO}) over $\k$ is a locally finite $\k$-linear abelian rigid monoidal category $\C$ such that the bifunctor $\otimes : \C \times \C \rightarrow \C$ is bilinear on morphisms and such that $\End_{\C}(\mathbf{1}) \cong \k$. A \textit{fusion category} is a finite semisimple tensor category.

\end{definition}

A \textit{symmetric tensor category} is a tensor category $\C$ that is also a symmetric monoidal category. If $\C$ is additionally finite and semisimple as an abelian category, then we call $\C$ a \textit{symmetric fusion category}. When $\C$ is a symmetric tensor category, we will use $c$ to denote the commutativity isomorphism (the braiding) and $c_{X, Y}$ to denote specifically the commutativity isomorphism $X \otimes Y \rightarrow Y \otimes X$. 

For the benefit of the reader, we give some simple examples of tensor categories.

\begin{example}

\begin{enumerate}

\item[(a)] The simplest examples of symmetric tensor categories are $\Vec$ and $\sVec$ which are, respectively, the categories of finite dimensional $\k$-vector spaces and finite dimensional $\k$-vector superspaces. To define superspaces we assume $p \not= 2$. Here, the braiding is just the swap map and the signed swap map, respectively.

\item[(b)] Similarly, the category of representations over $\k$ of a finite group is a symmetric finite tensor category over $\k$ with braiding given by the swap map.

\item[(c)] A slightly more complicated category is the universal Verlinde category in characteristic $p$, which we denote as $\Ver_{p}$, which is constructed as a quotient of the category of finite dimensional representations of $\mathbb{Z}/p\mathbb{Z}$ over $\k$ of characteristic $p$. The full details regarding the construction is given in section 2.1 as the construction is slightly technical. This category will be extremely important in the sequel as it plays a central role in the statement of the main theorems of this paper.

\end{enumerate}

\end{example}

\vspace{0.2cm}

We next recall the notion of a symmetric tensor functor. A \textit{symmetric tensor functor} between symmetric tensor categories is a faithful, exact, additive monoidal functor that is compatible with the commutativitiy isomorphisms. In the special case when the target category is $\Vec$ (resp. $\sVec$, resp. $\Ver_{p}$) we call the functor a \textit{fiber functor} (resp. \textit{super fiber functor}, resp. \textit{Verlinde fiber functor}.) In addition, if there exists a symmetric tensor functor from $\C$ to $\D$, we will say that $\C$ is fibered over $\D$.

An example of a fiber functor is the forgetful functor $\Rep_{\k}(G) \rightarrow \k$. On the other hand, $\Ver_{p}$ is an example of a symmetric fusion category that does not admit a fiber or super fiber functor.

For any symmetric tensor category $\C$, there is always a canonical full symmetric tensor functor $\Vec \rightarrow \C$, defined by sending the one-dimensional vector space to the unit object in $\C$. As a result, we can identify the subcategory of $\C$ consisting of objects that are direct sums of copies of $\mathbf{1}$, with the category of vector spaces over the base field. As a general principle, we call such objects trivial and commonly identify $X = V \otimes 1$, with the vector space $V =  \Hom_{\C}(1, X)$. 

\vspace{0.5cm}

\subsection{}

Deligne, in \cite{De}, introduced the notion of subexponential growth for a symmetric tensor category (though he does not explicitly use this term). We say that $\C$ has \textit{subexponential growth} if for any object $X$, there exists a positive number $c_{X}$ such that the $\length(X^{\otimes n}) \le c_{X}^{n}$ for all $n \ge 0.$ Deligne then proved that a symmetric tensor category $\C$ in characteristic 0 has subexponential growth if and only if it admits a super fiber functor (see \cite{De, EGNO, O} for more details.)

As a consequence of Deligne's theorem, every symmetric tensor category $\C$ of subexponential growth in characteristic zero has the Hilbert basis property, i.e., every finitely generated commutative algebra $A$ in $\C^{\ind}$ is Noetherian and for any object $X \in \C$, $A \otimes X$ is a Noetherian $A$-module (hence, so is every quotient). Additionally, in the case of a fusion category, Deligne's theorem implies that $\C = \Rep(G, u)$, the $\Z/2\Z$-graded representation category of a finite group $G$ with parity operator given by a fixed central element $u$ of order $2$ in $G$. Hence, by standard invariant theory arguments, for any finitely generated commutative ind-algebra $A$, the ring of invariants of $A$ is finitely generated and $A$ is finitely generated as a module over its invariants.

Recently, in \cite{O}, Victor Ostrik proved an extension of Deligne's theorem to symmetric fusion categories in characteristic $p > 0$. More precisely, he proved that any symmetric fusion category in characteristic $p > 0$ admits a Verlinde fiber functor to $\Ver_{p}$. The main goal of this paper is to apply these results of Ostrik to the Hilbert basis problem and the problem of finite generation of invariants. 

\vspace{0.5cm}

\subsection{}

We now turn to the results of this paper. Recall first that in any symmetric tensor category, we have a notion of a commutative algebra, which is an object equipped with an associative unital multiplication morphism that is invariant under the braiding. Similarly, we also have a notion of a module over a commutative algebra. Henceforth, fix a symmetric finite tensor category $\C$ over a field $\k$ of characteristic $p > 0$, with braiding $c$. Leaving aside the technical definitions of finite generation, Noetherianity and ind-completions $\C^{ind}$ till the end of section 2.1, our first main result is as follows.

\begin{theorem} \label{Hilb2} If $\C$ admits a Verlinde fiber functor, then every finitely generated commutative ind-algebra $A \in \C^{\ind}$ is Noetherian as an algebra.
\end{theorem}

If $\C$ is any category for which the conclusion of Theorem \ref{Hilb2} is true, then we say that $\C$ has the \textit{Hilbert basis property}. In particular, Ostrik's main result in \cite{O} gives us the following corollary of the above theorem.

\begin{corollary} \label{Hilb3} Symmetric fusion categories have the Hilbert basis property. 
\end{corollary}

\vspace{0.2cm}

Next, recall the definition of the Chevalley property. We say that a tensor category has the Chevalley property if the tensor product of simple objects is semisimple. Using Corollary \ref{Hilb3}, we prove the following result:

\begin{corollary} \label{Hilb4} Suppose $\C$ is a symmetric tensor category fibered over a symmetric tensor category with finitely many isomorphism classes of simple objects, in which the Chevalley property holds. Then, $\C$ has the Hilbert basis property. 
\end{corollary}

\begin{remark} 

The above corollaries show that the Hilbert basis property holds in a very general setting, which may lead one to think that it should hold for arbitrary symmetric tensor categories. To show that this is not the case, we give an example of a category for which the Hilbert basis property fails. In particular, this category will contain infinitely many isomorphism classes of simple objects, which shows that the finiteness assumption in Corollary \ref{Hilb4} is essential. 

Consider the Deligne category $\Rep S_{t}$ (see \cite{E} for details) associated to the symmetric group at $t \notin \Z_{\ge 0}$. This category is a semisimple symmetric tensor category but there are infinitely many isomorphism classes of simple objects. These isomorphism classes are correspond bijectively to partitions (of arbitrary size) and are hence labeled by Young diagrams. Let $X$ be the simple object corresponding to the one box diagram. This object can be thought of as an interpolation of the reflection representation for $S_{n}$ as $n$ grows large (\cite[2.2]{E}). One can check that $S(X)^{\inv}$ is isomorphic to the polynomial algebra in infinitely many variables and is hence not Noetherian as a module over itself, and thus is certainly not Noetherian as an algebra (we omit the proof as we do not use this result in the sequel). Thus, the Hilbert basis property fails for $\Rep S_{t}$ if $t \notin \Z$.

\end{remark}

\vspace{0.5cm}

\subsection{} We next turn to the problem of invariants. Again, leaving aside the formal definition of invariants of ind-objects in symmetric tensor categories to the end of section 2.1, we just note that, as in the classical setting, the invariants of any commutative ind-algebra in a symmetric tensor category form a commutative ind-subalgebra. This brings us to our next main result of the paper.

\begin{theorem} \label{Inv} Suppose $\C$ is a symmetric finite tensor category that admits a Verlinde fiber functor $F : \C \rightarrow \Ver_{p}$. Let $A \in \C^{\ind}$ be a finitely generated commutative algebra and let $A^{\inv}$ be its invariant subalgebra. Then, $A^{\inv}$ is finitely generated and $A$ is a finitely generated $A^{\inv}$-module.
\end{theorem}

Following Ostrik's result, this theorem has the following corollary.

\begin{corollary} Let $\C$ be a symmetric fusion category over $\k$ of characteristic $p$. Let $A$ be a finitely generated commutative algebra in $\C^{\ind}$. Then, $A^{\inv}$ is finitely generated and $A$ is a finitely generated $A^{\inv}$-module.
\end{corollary}

\begin{remark} Note that a statement in characteristic 0 analogous to the theorem above fails to be true. More precisely, there exists a category $\C$ fibered over $\sVec$ and a finitely generated commutative ind-algebra $A \in \C$ such that $A^{\inv}$ is not finitely generated. 

Let $\C$ be the category of complex representations of the $1$-dimensional purely odd Lie superalgebra, which is the category of super vector spaces

$$V = V_{+} \oplus V_{-}$$
with an odd linear map $D : V \rightarrow V$ that squares to $0$. Let $V$ be the representation  with basis $\{x, y, z\}$ $x$ even and $y, z$ odd and $D$ defined by sending $x$ to $y$ and $y, z$ to $0$. Let $A$ be the symmetric algebra of $V$ in $\C^{\ind}$. Then,

$$A = \mathbb{C}[x] \otimes \bigwedge(y, z).$$
Now, $A^{\inv} = \Ker(D)_{\text{even}}$ has basis over $\mathbb{C}$ given by $\{1, x^{n}yz : n  \ge 1\}$ and is hence not finitely generated. Additionally, $A$ is not finitely generated as a module over $A^{\inv}$. 

The reason this example fails to give a counterexample to the theorem in characteristic $p > 0$ is because in characteristic $p$, $x^{p}$ is also in $A^{\inv}$. Thus, $A^{\inv}$ becomes finitely generated and $A$ becomes a finitely generated module over $A^{\inv}$.

\end{remark}

\vspace{0.2cm}

The main strategy to proving Theorem \ref{Hilb2} is to prove Theorem \ref{Hilb2} for $\Ver_{p}$ and then to lift these properties along the Verlinde fiber functor. We then use this theorem, along with the construction of a Reynolds' operator, i.e., an $A^{\inv}$-module projection from $A$ onto $A^{\inv}$, to prove finite generation of invariants for fusion categories. 

Following this result, we turn to the question of finiteness of an algebra over its invariants and prove this using an analog of the Frobenius homomorphism in $\Ver_{p}$. This approach works uniformly for semisimple and non-semisimple categories. We then use this result to prove finite generation of invariants in an arbitrary tensor category, which also gives a different proof of the result in the fusion setting.

\begin{remark} 

In this paper, the above results are proved only in the case of tensor categories over algebraically closed fields. There are straightforward generalizations of these results to non-algebraically closed fields. However, a discussion of the generalized results would require a discussion of tensor categories over non-algebraically closed fields, which is outside the scope of the paper. 

\end{remark}

\vspace{0.5cm}

\subsection{}

After proving the above theorems for categories fibered over $\Ver_{p}$, we consider the case (in characteristic 2) of a category that does not admit a fiber functor to $\Ver_{2}$ (which is just $\Vec$ in characteristic $2$). Let $\k$ now be any field of characteristic $2$, not necessarily algebraically closed. Consider the commutative $\k$-algebra $D = \k[d]/(d^{2})$. This acquires the structure of a Hopf algebra with primitive $d$ (characteristic $2$ is required for the comultiplication to preserve the relation $d^{2} = 0$.)

We can give this Hopf algebra a triangular structure (see \cite[8.3]{EGNO}) via the $R$-matrix

$$R := 1 \otimes 1 + d \otimes d.$$
Then, $\Rep D$, with the above triangular structure, acquires the structure of a symmetric finite tensor category. 

There are two reasons this category is particularly interesting. First, $\Rep D$ does not fiber over $\Ver_{2}$ (which is just the category of $\k$-vector spaces). Hence not only is it not covered by the previous theorems in this paper but it is also a counterexample in characteristic $2$ to Conjecture 1.3 in \cite{O}. The second reason is that over fields of characteristic $2$, the ordinary definition of supervector spaces does not make sense as parity has no meaning. Instead, we can actually view $\Rep D$ as the correct category of supervector spaces over $\k$, as it is a nonsemisimple reduction of the category of supervector spaces over an ramified extension of the $2$-adics. Both the statements above will be elaborated upon and rigorously proved in section 2.3.

\vspace{0.2cm}

Let $\k$ now be an algebraically closed field of characteristic $2$. In section 2.3, we show that there is only one simple in $\Rep D$, the unit object $\mathbf{1}$. Hence, $\Rep D$ has the Chevalley property. Thus, by Corollary \ref{Hilb4}, if $\C$ is a symmetric finite tensor category fibered over $\Rep D$, then $\C$ has the Hilbert basis property. Our final goal of the paper is to show that an analog of Theorem \ref{Inv} also holds for such $\C$. Before stating the result, we introduce some notation, by using $\sVec_{2}$ to denote $\Rep D$ (justified by the paragraph above and section 2.3). The following theorem now holds.

\begin{theorem} \label{Inv2} Suppose $\C$ is a symmetric finite tensor category fibered over $\sVec_{2}$ and let $A$ be a finitely generated commutative ind-algebra in $\C$. Then, $A^{\inv}$ is finitely generated and $A$ is finite as a module over $A^{\inv}$.
\end{theorem}

The strategy for proving this theorem will be to reduce the problem to the classical case of commutative $\k$-algebras.

\vspace{0.5cm}

\subsection{Acknowledgements}

I would like to thank Victor Ostrik for his generous advice regarding his results and the results of this paper. I would also like to thank Augustus Lonergan and Nathan Harman for many useful conversations. Above all, I am deeply grateful to Pavel Etingof for suggesting the problem and providing invaluable advice. In particular, I thank him for explaining to me the definition of $\sVec_{2}$, the proof that it does not fiber over $\Vec$, and its intepretation in characteristic 2 (see Subsection 1.5). I also wish to thank him for suggesting the Deligne categories $\Rep S_{t}$ at $t \notin \Z_{\ge 0}$ as an example of a symmetric tensor category in which the Hilbert basis property fails to hold. I also wish to think the referees who reviewed this paper for their helpful comments regarding its organization and for pointing out errors in some of the proofs. This work was partially supported by the NSF grant DMS-1000113.

\section{Preliminaries}

\subsection{} We begin by recalling some technical facts about tensor categories and symmetric monoidal categories. We first recall two facts about tensor categories we will use in the sequel. 

\begin{enumerate}

\item As a consequence of rigidity (existence of left and right duals), the tensor product bifunctor $\otimes$ is biexact (see \cite[4.2]{EGNO} for a proof).

\item In any tensor category $\C$, the unit object $\mathbf{1}$ is simple (see \cite[4.3]{EGNO} for a proof).

\end{enumerate}

\vspace{0.2cm}

Next, note that the assumption that a monoidal category is symmetric also endows it with extra structure. Recall that a \textit{pivotal structure} on a rigid monoidal category $\C$ is a natural monoidal isomorphism $X \rightarrow X^{**}$. Such a structure allows us to define left and right traces of any morphism $X \rightarrow X$ and the pivotal structure is called \textit{spherical} if for every morphism, the left trace equals the right trace (see \cite[4.7]{EGNO}.) In particular, we can define for any object $X$, $\dim(X) = \mathrm{Tr}(\id_{X}).$ In the case of a symmetric monoidal category, we have a natural spherical structure given by

$$\begin{tikzpicture}
\matrix (m) [matrix of math nodes,row sep=6em,column sep=6em,minimum width=6em]
{ X & X \otimes X^{*} \otimes X^{**}  & X^{*} \otimes X \otimes X^{**} & X^{**}  \\};
\path[-stealth]
(m-1-1) edge node[auto] {$\id_{X} \otimes \coev_{X^{*}}$} (m-1-2)
(m-1-2) edge node[auto] {$c_{X, X^{*}} \otimes \id_{X^{**}}$} (m-1-3)
(m-1-3) edge node[auto] {$\ev_{X} \otimes \id_{X^{**}}$}(m-1-4);
\end{tikzpicture}$$
see \cite[9.9]{EGNO}. If $\C$ is now a tensor category equipped with a spherical structure, we say that $\C$ is a spherical tensor category. Recall that for spherical monoidal $\C$, a morphism $f : X \rightarrow Y$ is said to be $\textit{negligible}$ if for any morphism $u : Y \rightarrow X$, the trace of $fu$ is $0$ (see \cite[2.5]{O}). We will let $\N(X, Y)$ denote the space of negligible morphisms between $X$ and $Y$. We will also let $\N(\C)$ denote the objects $X$ in $\C$ for which $\id_{X}$ is negligible and call these objects negligible as well.

\vspace{0.2cm}

We now give a constructions related to symmetric tensor categories that will be extremely important in the sequel. As a special case of this construction, we will obtain the important universal Verlinde category $\Ver_{p}$ that is of central importance to the main theorems of this paper. 

\begin{example} \label{Ver}

To any rigid monoidal $\k$-linear category $\C$ endowed with a spherical structure, we can associate a rigid monoidal quotient category $\overline{\C}$ whose objects are the objects in $\C$ and whose morphism spaces are $\Hom_{\C}(X, Y)/\N(X, Y).$ Additionally, if $\C$ is in fact a tensor category, then the collection of all $\N(X, Y)$ is a tensor ideal, i.e., that if either $f$ or $g$ is negligible, then so are $fg$ and $f \otimes g.$ Hence, in this case, the tensor structure, braiding and spherical structure on $\C$ descend to a tensor structure, braiding and spherical structure on $\overline{\C}$, which is hence a spherical tensor category. In particular, this quotient category is semisimple and the simple objects are the indecomposables $X \in \C$ that are not in $\N(\C)$.

We now use this construction to define the universal Verlinde category $\Ver_{p}$ ($p > 0$) as the quotient of $\Rep_{\k}(\Z/p\Z)$ by the negligible morphisms. This category is semisimple with $p-1$ simple objects $L_{1}, \ldots, L_{p-1}$, where $L_{1} = \mathbf{1}$ and the fusion rules are given by

$$L_{r} \otimes L_{s} \cong \sum_{i=1}^{\min(r, s, p-r, p-s)} L_{|r -s| + 2i - 1}.$$
See \cite[3.2]{O} for more details. Later in the paper, we will give an alternative description of $\Ver_{p}$ in terms of tilting modules for $SL_{2}$ in characteristic $p$ that will prove to be more useful for the problems considered in this paper. 

\end{example}

\vspace{0.5cm}

We next discuss the connection of fiber functors to Hopf algebras.

\begin{remark} \label{Tannakian} \textbf{Tannakian Reconstruction}

Let $\C$ be a symmetric finite tensor category and let $\D$ be a symmetric tensor category such that there exists a symmetric tensor functor $F: \D \rightarrow \C$. Then, $\D$ is isomorphic as a symmetric tensor category to the category of finite dimensional comodules of $\mathrm{Coend}(F)$, which is a commutative Hopf algebra in $\C^{\ind}$. If $\D$ is finite, then $\mathrm{Coend}(F)$ is a commutative Hopf algebra in $\C$ and not just a commutative Hopf ind-algebra. 

For a proof of this fact in the case where $\C$ is the category of $\k$-vector spaces, see \cite[5.2, 5.4]{EGNO}. The proof in the case of general $\C$ is completely analogous.

\end{remark}

\vspace{0.2cm}

We next introduce some technical definitions that are used in the statement of the main theorems of this paper. We begin by defining the ind-completion of a symmetric tensor category.

\begin{definition} \label{ind}

Let $\C$ be a symmetric tensor category. By $\C^{\ind}$, we denote the ind-completion of $\C$, i.e., the closure of $\C$ under taking filtered colimits of objects in $\C$. Since the tensor product in $\C$ is exact, it commutes with taking filtered colimits and hence extends to a tensor product on $\C^{\ind}$. Additionally, naturality of the braiding implies that the braiding extends to a symmetric structure on $\C^{\ind}$. $\C^{\ind}$ is thus a symmetric $\k$-linear abelian monoidal category in which the tensor product structure $\otimes$ is exact (but it is neither rigid nor locally finite).

\end{definition}

A specific example of $\C^{\ind}$ that we will repeatedly use in the sequel is in the case where $\C$ is  a symmetric fusion category, i.e., when $\C$ is  finite and semisimple. In this case, the objects of $\C^{\ind}$ are precisely the (possibly infinite) direct sums of the simple objects in $\C$. 

As a matter of convention, if $\C$ is a symmetric finite tensor category,  when we use the word ``object", we will mean an object in $\C$, i.e., an object of finite length in $\C^{\ind}$ and we will use the term ind-object whenever referring to objects in $\C^{\ind}$ that may have infinite length. Sometimes, for emphasis, we will use the phrase ``actual object" to refer to the finite length objects.

\vspace{0.2cm}

We now define the notions of finite generation and Noetherianity for commutative ind-algebras in symmetric tensor categories and their modules.  Recall that in any symmetric tensor category, the notion of the symmetric algebra of an object is well defined as a commutative ind-algebra. For any object $X$ in $\C^{\ind}$, let $S(X)$ denote its symmetric algebra.

\begin{definition} We say that a commutative algebra $A$ in $\C^{\ind}$ is \textit{finitely generated} if there exists some actual object $X \in \C$ and a surjective morphism of algebras

$$S(X) \rightarrow A.$$
For an arbitrary commutative algebra $A \in \C^{\ind}$, we say that an $A$-module $M \in \C^{\ind}$ is finitely generated if there exists an object $X \in \C$ and a surjective morphism of $A$-modules

$$A \otimes X \rightarrow M$$
where the module structure on $A \otimes X$ comes from the module structure on $A$. 

\end{definition}

\begin{definition} \label{Noetherian} For a commutative ind-algebra $A$, we say that an $A$-module $M$ is \textit{Noetherian} if its $A$-submodules satisfy the ascending chain condition, i.e., that for any sequence of submodules

$$M_{0} \rightarrow M_{1} \rightarrow M_{2} \rightarrow \cdots$$
in which the morphisms are mono, there exists some $n$ such that for all $N \ge n$, the map $M_{N} \rightarrow M_{N+1}$ is an isomorphism. We say that $A$ is a \textit{Noetherian algebra} if all of its finitely generated modules are Noetherian.

\end{definition}

We will prove in section 2.2 that Noetherianity of a module is equivalent to finite generation of submodules. 

\begin{remark} 

Unlike the classical case,  in the general setting of an arbitrary tensor categories we don't know if an algebra which is Noetherian as a module over itself is necessarily Noetherian in the sense of Definition \ref{Noetherian}. 

\end{remark}

\vspace{0.2cm}

We end this subsection by formally defining the invariants of an ind-object in a symmetric tensor category $\C$.

\begin{definition} Let $X \in \C^{\ind}$. Then, we define the \textit{object of invariants} of $X$ to be the sum of all subobjects in $X$ that are isomorphic to $\mathbf{1}$. We denote this by $X^{\inv}$. Since a sum of simple objects is always a direct sum, $X^{\inv}$ is a direct sum of copies of $\mathbf{1}$, i.e., it is trivial. Thus, following the general principle as stated earlier, we view $X^{\inv}$ as an ordinary vector space by identifying it with $\Hom_{\C}(\mathbf{1}, X).$ 
\end{definition}

\begin{remark} 

If $\C$ is a symmetric finite tensor category with a Verlinde fiber functor $F$, then by Remark \ref{Tannakian} $\C$ is equivalent as a symmetric tensor category to the category of finite dimensional comodules of a commutative Hopf algebra $H$ in $\Ver_{p}$. In this case, for any comodule $X$ with coaction $\Delta$, $X^{\inv}$ is the sum of all subobjects $Y$ such that $\Delta|_{Y}$ coincides with the inclusion of $Y$ into $X \otimes H$ as $Y \otimes \mathbf{1}$ (with $\mathbf{1}$ in $\Ver_{p}$ identified with the unit object in $H$). 

\end{remark}

\vspace{0.5cm}

\subsection{} We now prove some preliminary facts about commutative ind-algebras in tensor categories. 

 Let $F : \C \rightarrow \D$ be a symmetric tensor functor between symmetric tensor categories. Since this is exact, it extends to a functor $F : \C^{\ind} \rightarrow \D^{\ind}$. 
 
\begin{lemma} \label{funct:symm} Let $X \in \C^{\ind}$. Then, 

$$F(S(X)) \cong S(F(X))$$
as objects in $\D^{\ind}$.

\end{lemma}

\begin{proof} This follows from the fact that symmetric tensor functors preserve tensor products and the commutativity isomorphisms.
\end{proof}

We next show that taking symmetric algebra of a direct sum gives a tensor product of symmetric algebras. First, we need to make a definition.

\begin{definition} Let $A$ and $B$ be commutative ind-algebras in $\C$. Then, we define a commutative multiplication on $A \otimes B$

$$A \otimes B \otimes A \otimes B \rightarrow A \otimes B$$
as $(m_{A} \otimes m_{B}) \circ (\id_{A} \otimes c_{B, A} \otimes \id_{B})$. 

\end{definition}

It is clear that $A \otimes B$ satisfies the same universal property in the category of commutative ind-algebras in $\C$ as it does in the standard case when $\C = \Vec$. Hence, we have the following lemma.

\begin{lemma} \label{symm:sum} Let $X, Y \in \C^{\ind}$. Then, we have a natural isomorphism

$$S(X \oplus Y) \cong S(X) \otimes S(Y)$$
of commutative ind-algebras.

\end{lemma}
\begin{proof} The proof follows by showing that both objects satisfy the same universal property. Let $\C'$ be the category of commutative ind-algebras in $\C$ and let $A \in \C'$. Then, we have natural isomorphisms,

\begin{align*}
\Hom_{\C'}(S(X \oplus Y), A) &\cong \Hom_{\C}(X \oplus Y, A) \\
&\cong \Hom_{\C}(X, A) \oplus \Hom_{\C}(Y, A) \\
&\cong \Hom_{\C'}(S(X), A) \oplus \Hom_{\C'}(S(Y), A) \\
&\cong \Hom_{\C'}(S(X)\otimes S(Y), A).
\end{align*}

\end{proof}

\vspace{0.2cm}

We next prove a claim we made in the introduction: Noetherianity is equivalent to finite generation of submodules. 

\begin{lemma} \label{Noeth:finite} Let $A$ be a commutative ind-algebra in $\C$. Let $M$ be an $A$-module. Then, $M$ is Noetherian if and only if every submodule $N$ of $M$ is finitely generated. 
\end{lemma}
\begin{proof} We first prove the forward direction. Suppose $M$ is Noetherian and assume for contradiction that there exists a submodule $N$ of $M$ that is not finitely generated. We inductively create an infinite ascending chain of finitely generated submodules $\{N_{i}\}$ of $N$ that does not terminate. 

Since $N$ is an ind-object in $\C$, it must contain an actual object $X_{0}$. Let $N_{0}$ be the submodule generated by $X_{0}$, i.e., let it be the submodule given by the image of $A \otimes X \rightarrow M$ under the action map. Then, as $X_{0}$ is a subobject of $N$, $N_{0}$ is a submodule of $N$ and is finitely generated. This is the first step of the construction. Now, suppose $N_{r}$ has been defined as a finitely generated submodule of $N$. Then, the inclusion of $N_{r}$ into $N$ is not an isomorphism as $N$ would then have been finitely generated. Hence, the cokernel of this inclusion is not $0$ and thus contains a nonzero object $Y \in \C$. Take the submodule generated by this object and call it $\overline{N_{r+1}}$, which is a submodule of $N/N_{r}$. Hence, we can find a submodule $N_{r+1}$ of $N$ such that the inclusion of $N_{r}$ into $N$ factors properly through $N_{r+1}$, and the quotient $N_{r+1}/N_{r}$ is $\overline{N_{r+1}}$. 

We need to show that $N_{r+1}$ is finitely generated. We know that $N_{r}$ is finitely generated and so is $N_{r+1}/N_{r}$. Let $X_{r}$ and $Y$ be actual objects in $\C$ with $X_{r} \subseteq N_{r}$ and $Y \subseteq \overline{N_{r+1}}$ such that the natural maps

$$A \otimes X_{r} \rightarrow N_{r}$$
and 

$$A \otimes Y \rightarrow \overline{N_{r+1}}$$
are epimorphisms. We claim that there exists some subobject $Z$ of $N_{r+1}$, with $Z \in \C$, such that the projection from $Z$ to $\overline{N_{r+1}}$ contains $Y$. This is because the direct limit of the images of subobjects of $N_{r+1}$ under the projection map is $\overline{N_{r+1}}$ and hence any object of finite length must be contained in the direct limit of some finite subset of these images. So, if $Z_{1}, \ldots, Z_{n}$ are subobjects of $N_{r+1}$, the sum of whose images contains $Y$, then we take $Z = \sum_{i=1}^{n} Z_{i}$.

We now claim that $X_{r}$ and $Z$ together generate $N_{r+1}$. Let $X_{r+1} = Z + X_{r}$. Then, the submodule generated by $X_{r+1}$ contains $N_{r}$ (as it contains $X_{r}$) and surjects under projection to $\overline{N_{r+1}}$ (as it contains $Z$). Hence, $X_{r+1}$ generates $N_{r+1}$, which is therefore finitely generated.

We now prove the reverse direction. Suppose every submodule of $M$ is finitely generated. Let

$$M_{0} \rightarrow M_{1} \rightarrow \cdots$$
be a sequence of monomorphisms of $A$-submodules of $M$. Then, since $\C^{\ind}$ is closed under filtered colimits, we can take the colimit (in this case the union) of this sequence to get an ind-subobject $M'$ of $M$. Additionally, since the $A$-module structure on $M_{i}$ commutes with the morphisms that we take the colimit of, $M'$ acquires a natural structure of an $A$-module. Hence, by the assumption that $A$-submodules of $M$ are finitely generated, there exists an object $X \in \C$ and an epimorphism $A \otimes X \rightarrow M'$ of $A$-modules.

But now, such a morphism has to come from a morphism in $\C^{\ind}$ from $X$ to $M'$. As $X$ is an actual object in $\C$, its image in $M'$ has finite length and hence must lie in some $M_{i}$ (as otherwise, taking the intersection of the image with $M_{i}$ gives an infinite ascending chain of subobjects of $X$ that does not stabilize, which cannot exist). Hence, as $M_{i}$ is an $A$-module, the image of $A \otimes X \rightarrow M'$ lies in $M_{i}$ and hence the inclusion of $M_{i}$ into $M'$ is an isomorphism. Hence, for all $N \ge i$, the inclusion $M_{N} \rightarrow M_{N+1}$ is an isomorphism. This proves that $M$ is Noetherian.

\end{proof}

\vspace{0.2cm}

Recall that any symmetric tensor category $\C$ contains the category of vector spaces as the symmetric tensor subcategory consisting of direct sums of copies of $\mathbf{1}$. Hence, for any $n$, $\k[x_{1}, \ldots, x_{n}]$ can be viewed as an ind-algebra in $\C^{\ind}$ via this embedding of $\Vec$ into $\C$. Our final goal of this section is to prove that tensoring with this polynomial algebra preserves Noetherianity. The proof of this assertion is very similar to the proof of the classical Hilbert basis theorem but is stated in a categorical manner. We give the statement and proof below. If the proof seems complicated, just translate the statements to the case when everything is actually a vector space and it will make sense.

\begin{proposition} \label{poly} Let $A$ be a commutative ind-algebra in $\C$. Let $M$ be a Noetherian $A$-module. Then, $\k[x_{1}, \ldots, x_{r}] \otimes M$ has a natural structure of a Noetherian $\k[x_{1}, \ldots, x_{r}] \otimes A$-module. 
\end{proposition}

\begin{proof} By induction on $r$, we may assume $r = 1$. We can then write

$$\k[x] \otimes M = M \oplus xM \oplus x^{2}M \oplus \cdots$$
Here $xM$ can be viewed as $x \otimes M$ or as the image of $M$ under the action of $x$. This clearly has a natural structure of a $\k[x] \otimes A$-module defined as $\id_{\k[x]} \otimes c_{A, \k[x]} \otimes \id_{M}$ followed by componentwise action. Using Lemma \ref{Noeth:finite}, we will show that this is Noetherian by showing that any $\k[x] \otimes A$-submodule $N$ of $\k[x] \otimes M$ is finitely generated. 

For $n \ge 0$, define $M_{n} = M \oplus \cdots \oplus x^{n}M$ and let $\pi_{n}$ be the projection $\k[x] \otimes M \rightarrow x^{n}M \cong M$. For each finitely generated $A$-submodule $X$ of $N$, we define the associated object of leading coefficients $LC(X)$ , which will be an $A$-submodule of $M$. Any such finitely generated submodule must be contained in $M_{n}$ for some $n$. Let $n(X)$ be the minimal such $n$ for $X$ and define $LC(X) = \pi_{n(X)}(X) \subseteq x^{n(X)}M$, which we identify with $M$ by the multiplication isomorphism $x^{n(X)} :  M \rightarrow x^{n(X)}M.$ 

Define $LC_{N}$ to be the sum of $LC(X)$ over all finitely generated $A$-submodules $X$ of $N$. This is an $A$-submodule of $M$ and is hence finitely generated by Noetherianity of $M$. If $A \otimes Z \rightarrow LC_{N}$ is a surjection of $A$-modules with $Z \in \C$, then this comes from a morphism $Z \rightarrow LC_{N}$ in $\C$. Let $X_{1}, \ldots, X_{m}$ be finitely generated $A$-submodules of $N$ such that $\sum_{i=1}^{m} LC(X_{i})$ contains the image of the morphism from $Z$ (some such finite list must exist as $Z$ has finite length). Then, $LC_{N} =  \sum_{i=1}^{m} LC(X_{i})$. 

Let $d_{i} = n(X_{i})$ and let $d$ be the maximum of the $d_{i}$. Finally, define

$$B: =  k[x](N \cap M_{d}).$$
Clearly $B \subseteq N$ and by choice of $d$, we also have

$$\sum_{i=1}^{m} k[x]X_{i} \subseteq B.$$
We claim that $B = N$. Suppose for contradiction that $B \not = N$. Since $N \in \C^{\ind}$, it is the sum of all the objects it contains. Hence, we can find some object $Y \in \C$ that is a subobject of $N$ but is not contained in $B$. Taking the $A$-submodule of $N$ generated by $Y$, we see that there exist finitely generated $A$-submodules of $N$ that are not contained in $B$. Let $N'$ be such a submodule such that $h = n(N')$ is minimal amongst all such submodules. Note that $h > d$ as otherwise $N' \subseteq N \cap M_{d} \subseteq B$.

Consider now the finitely generated $A$-submodule $B \cap M_{h}$ of $B$. Note that 

$$\sum_{i=1}^{m}x^{h-d_{i}} X_{i} \subseteq B \cap M_{h}$$
and hence, $LC_{N} \subseteq \pi_{h}(B \cap M_{h})$. From this, it follows that

$$\pi_{h}(N') = LC(N') \subseteq LC(B \cap M_{h}) = \pi_{h}(B \cap M_{h}) = \sum_{i=1}^{m} LC(X_{i}) = LC_{N}.$$ 

We show this implies the existence of a finitely generated $A$-submodule $N''$ of $(B \cap M_{h}) + N'$ that is not contained in $B \cap M_{h}$ and has $n(N'') < n(N')$. For this purpose, consider the inclusion of $B \cap M_{h}$ into $(B \cap M_{h}) + N'$. Since $B \cap M_{h}, N' \subseteq M_{h}$, we have a commutative diagram

$$\begin{tikzpicture}
\matrix (m) [matrix of math nodes,row sep=3em,column sep=3em,minimum width=3em]
{ 0 & B \cap M_{h-1} & B \cap M_{h} & \pi_{h}(B) & 0 \\
 0 & ((B \cap M_{h}) + N') \cap M_{h-1} & (B \cap M_{h}) + N' & \pi_{h}((B \cap M_{h}) + N') & 0 \\};
\path[-stealth]
(m-1-1) edge (m-1-2)
(m-1-2) edge (m-1-3)
(m-1-3) edge (m-1-4)
(m-1-4) edge (m-1-5)
(m-2-1) edge (m-2-2)
(m-2-2) edge (m-2-3)
(m-2-3) edge (m-2-4)
(m-2-4) edge (m-2-5)
(m-1-2) edge node[auto] {$\alpha$} (m-2-2)
(m-1-3) edge node[auto] {$\beta$} (m-2-3)
(m-1-4) edge node[auto] {$\gamma$} (m-2-4);
\end{tikzpicture}$$
where the rows are exact and the vertical maps are induced by the inclusion of $B \cap M_{h}$ into $(B \cap M_{h}) + N'$. All three vertical maps are monomorphisms. $\gamma$ is an epimorphism based on the argument above.  But, by choice of $N'$, $\beta$ is not an epimorphism. Hence, $\alpha$ cannot be an epimorphism by the five lemma. Hence, we can find a finitely generated nonzero $A$-submodule $N''$ of $((B \cap M_{h}) + N') \cap M_{h-1} \subseteq N \cap M_{h-1}$ that is not contained in $B \cap M_{h}$. Since $N''$ is contained in $M_{h-1} \subseteq M_{h}$, this implies that $N''$ is not contained in $B$. This contradicts the minimality of $n(N')$. Hence, $N=B$ is hence finitely generated as a $\k[x] \otimes A$-module.

\end{proof}

\vspace{0.5cm}

\subsection{} We end section 2 by proving the claims made in section 1.5. Let $\k$ be a field of characteristic $2$ and let $D$ be the Hopf algebra $\k[d]/(d^{2})$ defined as in section 1.5, with the $R$-matrix

$$R = 1 \otimes 1 + d \otimes d.$$

We first show that $\Rep D$ does not fiber over $\Ver_{2}$, which is just the category of vector spaces over $\k$. This proof was suggested by Victor Ostrik. Note that objects in $\Rep D$ are $\k$-vector spaces $V$ equipped with endomorphisms $d : V \rightarrow V$ such that $d^{2} = 0$.  Thus, $\Rep D$ has two indecomposables: the one dimensional vector space $\mathbf{1}$ with $d = 0$, and the two dimensional vector space $W$ with $d$ the strictly upper triangular matrix

$$\left(\begin{array}{cc}
0 & 1\\
0 & 0 \end{array}\right).$$
Now, $W$ is a self extension of $\mathbf{1}$. Thus, if there existed a fiber functor $F: \Rep D \rightarrow \Ver_{2}$, then, $F(W) \cong \mathbf{1} \oplus \mathbf{1}$. Hence, if $U$ is the copy of $\mathbf{1}$ that is a subobject of $W$, then the existence of a fiber functor would imply that the natural map

$$S(U) \rightarrow S(W)$$
is a monomorphism, as the natural map $S(F(U)) \rightarrow S(F(W))$ is a monomorphism. But we can show that this fails to be true. $W$ has a basis $\{x, y\}$ with $d(y) = 0$, $d(x) = y$. Then, $\k\{y\} = U$. Using the definition of the $R$-matrix, we can see that for $a, b \in S(X)$,

$$[a, b] = d(a)d(b).$$
Hence, in $S(X)$,

$$0 = [x, x] = d(x)d(x) = y^{2}.$$
Thus, the natural map $S(U) \rightarrow S(W)$ is not a monomorphism and $\Rep D$ is not fibered over $\Ver_{2}$.

\begin{remark}

Note that algebras of the form $S(W \otimes V)$ (with $W$ the 2 dimensional indecomposable in $\Rep D$ as above and $V$ a multiplicity space) appeared as $\widetilde{\Omega}(V)$ in the study of lower central series of free associative algebras in \cite{BEJKL}. 

\end{remark}

\vspace{0.2cm}

We next show that for $\k = \F_{2}$, the field with $2$ elements, $\Rep D$ can be constructed as a nonsemisimple reduction of the category of supervector spaces over a ramified extension of $\Q_{2}$, the $2$-adics. Hence, over any field of characteristic $2$, we can consider $\Rep D$ to be a nonsemisimple analog of supervector spaces (by extension of scalars).

Let $\F = \Q_{2}[\sqrt{2}]$ and let $\O = \Z_{2}[\sqrt{2}]$ be the ring of integers in $\F$. Consider the group algebra $H$ of $\Z/2\Z = \langle 1, g \rangle$ (over $\F$) with $R$-matrix

$$R = \frac{1}{2}(1 \otimes 1 + 1 \otimes g + g \otimes 1 - g \otimes g).$$
Then, $\Rep H$ is the category of supervector spaces over $\F$. 

Let $b = g - 1$. Then, we can rewrite $R$ as

$$R = 1 \otimes 1 - \frac{2}b \otimes b.$$
Let $a =\frac{1}{\sqrt{2}} b$. Then, $a^{2} = -\sqrt{2}a$,

$$R = 1 \otimes 1 - a \otimes a$$
and

$$\Delta(a) = a \otimes 1 +1 \otimes a + \sqrt{2}(a \otimes a).$$

This defines an order in $H$ over $\O$ and reducing this order modulo the maximal ideal in $\O$ (the ideal generated by $\sqrt{2}$) gives us our Hopf algebra $D$ over $\F_{2}$. Thus, $\Rep D$ (over $\F_{2}$) is a nonsemisimple reduction of the category of supervector spaces over $\F$.

\vspace{0.5cm}

\section{Proof of Theorems \ref{Hilb2} and \ref{Inv} for $\C = \Ver_{p}$.}

\subsection{Description of finitely generated algebras in $\Ver_{p}$.} From this section on, $p$ will be a prime. Before we begin the proof of the theorems for $\Ver_{p}$, we briefly describe a different construction of this category. This construction is fairly involved so we refer the reader to \cite[3.2, 4.3]{O} and the additional references \cite{GK, GM} contained within for more details. Consider the category of rational $\k$-representations of a simple algebraic group $G$ of Coxeter number less than $p$, the characteristic of $\k$. This has a full subcategory consisting of tilting modules, which are those representations $T$ such that $T$ and its contragredient both have filtrations whose composition factors are Weyl modules $V_{\lambda}$ corresponding to dominant integral weights $\lambda$. This is only a Karoubian symmetric monoidal category and not a symmetric tensor category but since it is symmetric monoidal, it is still equipped with a spherical structure. Hence, we can still take its quotient by negligible morphisms. This gives us a symmetric fusion category which we denote $\Ver_{p}(G)$, the Verlinde category corresponding to $G$.

\vspace{0.5cm}

We are now ready to prove the following proposition.

\begin{proposition} \label{Ver:finite} Let $X$ be an object in $\Ver_{p}$ and let $n$ be the multiplicity of $\mathbf{1}$ in $X$. Then, 

$$S(X) \cong \k[x_{1}, \ldots, x_{n}] \otimes Y$$
for some object $Y \in \Ver_{p}.$ 

\end{proposition}

\begin{proof}

For $p = 2$, $\Ver_{p}$ is just the category of vector spaces over $\k$ and the proposition is hence trivial. So, we assume $p > 2$. It is known (see for example \cite[4.3]{O}) that $\Ver_{p}$, as described in the Example \ref{Ver}, is equivalent as a symmetric tensor category to $\Ver_{p}(SL_{2})$ and that $L_{i}$ corresponds to $V_{i-1}$ (where $i-1$ is the highest weight of the corresponding tilting module) under the equivalence. Additionally, for $p \ge 3$ and $n > 1$, Ostrik also showed in \cite[4.3]{O} that the fiber functor from $\Ver_{p}(SL_{n})$ to $\Ver_{p}$ takes the standard module of $SL_{n}$ to $L_{n}$ (note that here $n \le p - 1$).

Thus, by Lemma \ref{funct:symm}, we can compute the symmetric algebra of $L_{n}$ by computing the symmetric algebra of the standard module $V$ of $SL_{n}$. Now, we have a monoidal functor from the Karoubian monoidal category of tilting modules of $SL_{n}$ to $\Ver_{p}(SL_{n})$. In general, symmetric powers do not exist in a Karoubian category in positive characteristic, as quotients do not exist. However, for $n < p$, for any object $X$, we can identify $S^{n}(X)$ as a direct summand of $X^{\otimes n}$ and direct summands do exist in a Karoubian category. Hence, by a slight variation of Lemma \ref{funct:symm}, $S^{p-n +1}(V)$ is just the image under the quotient functor of the tilting module $S^{p-n+1}(V)$ in the category of $SL_{n}$-representations. 

Now, the dimension of $S^{r}(V)$ is $\binom{n+r-1}{n-1}$, so plugging in $r = p-n+1$ gives us $\binom{p}{n-1} = 0$. Hence, $S^{p-n+1}(V)$ is negligible and thus goes to $0$ under the quotient functor to $\Ver_{p}(SL_{n})$. For $k \ge p - n+2$, $S^{k}(V)$ is a quotient of $V^{\otimes k - p+n-1} \otimes S^{p-n+1}(V)$ in $\Ver_{p}(SL_{n})$ and is hence also $0$. Thus, the symmetric algebra of $V$ and therefore the symmetric algebra of $L_{n}$ is an actual object and not an ind-object, for every $1 < n \le p-1$.

Now, suppose $X \in \Ver_{p}$ is of the form $n \mathbf{1} \oplus Z$, where $Z$ is a direct sum of copies of $L_{2}, \ldots, L_{p-1}$. Using Lemma \ref{symm:sum} and the argument in the previous paragraph, we see that $Y = S(Z)$ has finite length, i.e., it is an actual object in $\Ver_{p}$. Hence, by Lemma \ref{symm:sum} again, 

$$S(X) \cong S(n\mathbf{1}) \otimes Y = \k[x_{1}, \ldots, x_{n}] \otimes Y$$
with $Y \in \Ver_{p}$, as desired. 

\end{proof}

\vspace{0.5cm}

\subsection{Proof of Theorems \ref{Hilb2} and \ref{Inv} for $\Ver_{p}$.} We are now ready to finish the proof of these two theorems for $\Ver_{p}$. The first step in the proof is a reduction to the graded case. This is immediate for Theorem \ref{Hilb2} because any finitely generated $A$-module is also a finitely generated $S(X)$-module if $S(X)$ surjects onto $A$. The same holds for Theorem \ref{Inv}. Suppose $A$ is a finitely generated commutative ind-algebra in $\Ver_{p}$. Then, we have an epimorphism

$$\phi : S(X) \rightarrow A$$
for some $X \in \Ver_{p}$. Additionally, as $\Ver_{p}$ is semisimple, $\phi$ also restricts to a surjection $S(X)^{\inv} \rightarrow A^{\inv}$ by definition of the invariant subobject. Suppose now that we have proved the theorem for $S(X)$. Then, $S(X)^{\inv}$ is finitely generated and hence so is $A^{\inv}$. Additionally, there exists some object $Y \in \Ver_{p}$ and an epimorphism

$$\psi: S(X)^{\inv} \otimes Y \rightarrow S(X)$$
of $S(X)^{\inv}$-modules. Composing with $\phi$ and noting that $\ker(\phi|_{S(X)^{\inv}}) \otimes Y$ goes to $0$, we get an epimorphism

$$A^{\inv} \otimes Y \rightarrow A$$
of $A^{\inv}$-modules. Hence, we have reduced the proof of the theorems to the case where $A = S(X)$ for some $X \in \Ver_{p}$. But, by Proposition \ref{Ver:finite}, 

$$S(X) \cong \k[x_{1}, \ldots, x_{r}] \otimes Y$$
for some $r \ge 0$ and some object $Y$ in $\Ver_{p}$.  We now apply Lemma \ref{poly}. Since $Y$ has finite length, we see that $S(X)$ is finitely generated over $\k[x_{1}, \ldots, x_{r}]$ in $\Ver_{p}$ and is hence a Noetherian algebra. Additionally, the invariants of $S(X)$ are finitely generated as a module over $\k[x_{1}, \ldots, x_{r}]$ and are hence certainly finitely generated. Finally, as $S(X)$ is finitely generated over $\k[x_{1}, \ldots, x_{r}]$ it is finitely generated over its invariants. $\hfill\square$.

\vspace{0.5cm}

\subsection{Proof of Theorem \ref{Hilb2} for general $\C$ fibered over $\Ver_{p}$.}

Let $\C$ now be a symmetric finite tensor category that admits a Verlinde fiber functor $F$. Extend $F$ canonically to a fiber functor $\C^{\ind} \rightarrow \Ver_{p}^{\ind}$.  Let $A$ be a finitely generated commutative ind-algebra in $\C$. Then, $F(A)$ is a finitely generated commutative ind-algebra in $\Ver_{p}$ (by Lemma \ref{funct:symm}). Similarly, if $M$ is a finitely generated $A$-module, then $F(M)$ is a finitely generated $F(A)$-module. Suppose we have an ascending chain

$$M_{0} \subseteq \cdots$$
of submodules of $M$. Then, for large enough $N$, the inclusion $F(M_{N}) \rightarrow F(M_{N+1})$ is an isomorphism. Hence, the cokernel of this morphism is $0$. But as $F$ is exact, the cokernel of this morphism is $F(\coker (M_{N} \rightarrow M_{N+1}))$. As $F$ is faithful, $M_{N} \rightarrow M_{N+1}$ must be an epimorphism and hence an isomorphism for large enough $N$. Hence, $M$ is a Noetherian $A$-module. Since this holds for all finitely generated $A$-modules $M$, $A$ is a Noetherian commutative algebra in $\C$. $\hfill \square$

\vspace{0.5cm}

\subsection{Proof of Corollary \ref{Hilb4}} 

The proof of Theorem \ref{Hilb2} for categories fibered over $\Ver_{p}$ immediately implies Corollary \ref{Hilb3}. We use this corollary to prove Corollary \ref{Hilb4}, i.e., that the Hilbert basis property holds for symmetric tensor categories fibered over symmetric tensor categories with finitely many isomorphism classes of simple objects, in which the Chevalley property holds. Let $\C'$ be such a category and let $\C$ be the symmetric tensor category (with the finiteness and Chevalley properties) that $\C'$ fibers over. Using the same argument as in the previous subsection, to show that $\C'$ has the Hilbert basis property it suffices to show that $\C$ has the Hilbert basis property.

Recall that the Chevalley property holds in a tensor category if the tensor product of simple objects is semisimple. Thus, $\C_{\text{ss}}$, the subcategory consisting of semisimple objects, is a symmetric tensor subcategory of $\C$. Additionally, the finite number of isomorphism classes of simple objects in $\C$ ensures that $\C_{\text{ss}}$ is a finite abelian category and is hence a symmetric fusion category. Hence, by Corollary \ref{Hilb3}, $\C_{\text{ss}}$ has the Hilbert basis property. 

To prove Corollary \ref{Hilb4} for $\C$, it again suffices to prove that $S(X)$ is Noetherian as an algebra, with $X$ an actual object in $\C$. Now, $X$ has a socle filtration

$$X = X_{m} \supseteq X_{m-1} \supseteq \cdots \supseteq X_{0} = 0$$
whose associated graded object $\gr X$ lies in $\C_{\text{ss}}$. Additionally, this socle filtration induces an ascending filtration on $S(X)$ and we have a canonical epimorphism

$$S(\gr X) \rightarrow \gr S(X).$$
Since $\C_{\text{ss}}$ is closed under the tensor product, $S(\gr X)$ is a finitely generated commutative ind-algebra in $\C_{\text{ss}}$. Thus, by Corollary \ref{Hilb3}, $S(\gr X)$ and thus $\gr S(X)$ are both Noetherian commutative algebras.

To now prove that $S(X)$ is Noetherian, it suffices to prove that $S(X) \otimes Y$ is a Noetherian $S(X)$-module for any actual object $Y$ in $\C$. Now, the filtration on $S(X)$ induces a filtration of $S(X)$-modules on $S(X) \otimes Y$. The associated graded module $\gr (S(X) \otimes Y)$ is equipped with a canonical epimorphism of $\gr S(X)$-modules

$$\gr (S(X))\otimes Y \rightarrow \gr (S(X) \otimes Y)$$
and hence the latter is finitely generated over $\gr S(X)$ and thus Noetherian as a $\gr S(X)$-module. But this implies that $S(X) \otimes Y$ is Noetherian as an $S(X)$-module, as desired. Hence, $S(X)$ is a Noetherian algebra and the Hilbert basis property holds for $\C$. 

\vspace{0.5cm}

\section{Proof of Finite Generation of Invariants for Fusion $\C$}

\subsection{A Reynolds' operator.} Let $\C$ be a symmetric fusion category. Then, $\C$ has the Hilbert basis property by Corollary \ref{Hilb3}. We will use this to prove finite generation of invariants for $\C$. The first step in the proof is a reduction to the graded case. This is done in exactly the same manner as in subsection 3.2 (the $\Ver_{p}$ case).

Let $A$ be a finitely geerated commutative algebra in $\C^{\ind}$. We now define a suitable projection from $A$ onto $A^{\inv}$. 

\begin{definition} If $A$ is a commutative ind-algebra in $\C$, a \textit{Reynolds' Operator} on $A$ is an $A^{\inv}$-module map $\rho : A \rightarrow A^{\inv}$ that is the identity on $A^{\inv}$. 
\end{definition}

Since $\C$ is a fusion category, then $A = A^{\inv} \oplus A^{\not=\mathbf{1}}$ and hence the canonical projection onto $A^{\inv}$ is a Reynolds' operator. 

\vspace{0.2cm}

Using the Reynolds' operator, we can now show that $A^{\inv}$ is Noetherian as a module over itself (but a priori not as an algebra, although this will be true as a consequence of finite generation and Corollary \ref{Hilb3} applied to $\C$.) Suppose we have an ideal $I$ in $A^{\inv}$. Let $AI$ be the ideal it generates in $A$. Then, it is immediate that $\rho(AI) = I$ as $AI = I \oplus AI^{\not=\mathbf{1}}$. Hence, as $AI$ (the image of $A \otimes I$ in $A$ under multiplication) is a finitely generated ideal of $A$ (Corollary \ref{Hilb3} and Lemma \ref{Noeth:finite} applied to $A$), $I = \rho(AI)$ is a finitely generated ideal. Hence, $A^{\inv}$ is Noetherian as a module over itself.

But $A^{\inv}$ is a trivial object, i.e., it is a direct sum of copies of $\mathbf{1}$, and hence can be viewed as an ordinary algebra in the category of vector spaces over $\k$. In particular, as we have reduced the problem to the case where $A$ was a $\Z_{\ge 0}$ graded algebra, $A^{\inv}$ is also a $\Z_{\ge 0}$ graded algebra. Thus, $A^{\inv}$ is a $\Z_{\ge 0}$ graded commutative ordinary algebra in the category of vector spaces that is Noetherian as a module over itself and therefore finitely generated as an algebra, as any set of generators of the ideal $A^{\inv}_{>0}$ also generates $A^{\inv}$ as an algebra. This proves the first part of Theorem \ref{Inv}, the finite generation of invariants, for symmetric fusion categories $\C$.

\vspace{0.1cm}

We next turn to the question of finiteness of an algebra over its invariants. This result will be independent of the techniques developed in this section and will hold for any $\C$ that has a Verlinde fiber functor.

\vspace{0.5cm}

\section{Proof of Theorem \ref{Inv} for General $\C$ Fibered over $\Ver_{p}$}

\subsection{} In this section, we relax the assumption that $\C$ is a fusion category and just assume that $\C$ is a symmetric finite tensor category fibered over $\Ver_{p}$. We will give a proof of Theorem $\ref{Inv}$ for such $\C$ that will also give a different proof of finite generation of invariants for fusion $\C$. 

Let $\C$ be a symmetric finite tensor category with a Verlinde fiber functor $F$. Let $A$ be a finitely generated commutative ind-algebra in $\C$. We wish to prove that $A$ is finitely generated as a module over $A^{\inv}$ and that $A^{\inv}$ is a finitely generated algebra. We do so by proving a different result: $F(A)^{\inv}$ is finitely generated as a module over $A'$, a finitely generated $\k$-subalgebra of $F(A^{\inv})$ (which we view as a $\k$-algebra as it is trivial). We subsequently show that this result implies Theorem \ref{Inv} for $\C$.

\vspace{0.5cm}

\subsection{} \textbf{A Frobenius operator.} We prove this result using the Frobenius endomorphism for characteristic $p$ commutative algebras. We begin by introducing some notation. For any $V \in \Ver_{p}$, we write

$$V = \sum_{i} V_{i} \otimes L_{i}$$
where where $L_{1}, \ldots, L_{p-1}$ are the simples in $\Ver_{p}$ and $V_{i} = \Hom(L_{i}, V)$ is a vector space giving the multiplicity. To simplify notation, we write

$$F(A) = \bigoplus_{i=1}^{p-1} A_{i} \otimes L_{i}.$$
Using the standard abuse of notation for trivial objects, we identify $A_{1}$ with $F(A)^{\inv}$. Since $F$ is a fiber functor from $\C$ to $\Ver_{p}$, we can identify $\C$ with the category of finite dimensional comodules over a commutative Hopf algebra $H$ in $\Ver_{p}$ (This is an actual object in $\Ver_{p}$ by finiteness of $\C$. See Remark \ref{Tannakian} for a reference). Then, we have a comodule morphism

$$\rho: F(A) \rightarrow F(A) \otimes H$$
which is a homomorphism of commutative ind-algebras in $\Ver_{p}$. Now, $F(A)^{\inv}$ is a commutative $\k$-algebra and hence we have a Frobenius ring homomorphism $x \mapsto x^{p}$ from $F(A)^{\inv}$ to itself. Let $A_{1}^{p}$ denote the image of this endomorphism. This is a commutative ind-algebra in the category of $\k$-vector spaces (as $\k$ is algebraically closed). Similarly, using the notation $H_{i} = \Hom(L_{i}, H)$ we introduced above, we define $H_{1}^{p}$ to be the image of the Frobenius endomorphism in $H_{1}$. Then, $H_{1}^{p}$ is a commutative $\k$-subalgebra of $H_{1}$. 

\vspace{0.2cm}

We claim that $H_{1}^{p}$ is a Hopf subalgebra of $H$. It is clear that $H_{1}^{p}$ is closed under the antipode (as the antipode is an algebra homomoprhism in $H$ and must preserve $H_{1}$, the $\mathbf{1}$-isotypic component of $H$). Thus, we just need to show that it is closed under comultiplication. Let $\Delta$ be the comultiplication morphism in $H$. We need to show that $\Delta(H_{1}^{p}) \subseteq H_{1}^{p} \otimes H_{1}^{p}$. Note that $\Delta$ gives us a homomorphism of $\k$-algebras

$$H_{1} \rightarrow (H \otimes H)_{1}$$
where

$$H \otimes H = \bigoplus_{i=1}^{p-1}(H \otimes H)_{i} \otimes L_{i}$$
Using the fusion rules for $\Ver_{p}$, we see that $\mathbf{1}$ is a subobject of $L_{i} \otimes L_{j}$ if and only if $i = j$, and then it has multiplicity $1$. Hence, 

$$(H \otimes H)_{1} = \bigoplus_{i=1}^{p-1} H_{i} \otimes H_{i}.$$
Let $x \in H_{1}^{p}$ with $x = y^{p}$. Since $\Delta$ is an algebra homomorphism, we see that if 

$$\Delta(y) = \sum_{i=1}^{p-1}\sum_{j=1}^{n_{i}} y_{ij} \otimes y'_{ij}$$
where $y_{ij}, y'_{ij}$ are elements in $H_{i}$, then

$$\Delta(x) = \sum_{j} y_{ij}^{p} \otimes (y'_{ij})^{p}.$$
But now, for each $i$, $y_{ij} \otimes y'_{ij}$ is an element in a copy of $\mathbf{1}$ sitting inside $L_{i} \otimes L_{i} \subseteq H$ and hence $y_{ij}^{p} \otimes (y'_{ij})^{p}$ is an element in a copy of $\mathbf{1}$ sitting inside the image of the corresponding componentwise multiplication morphism $S^{p}(L_{i}) \otimes S^{p}(L_{i}) \rightarrow H \otimes H$. In the proof of Proposition \ref{Ver:finite}, we showed that $S^{n}(L_{i})$ is $0$ for any $i > 1$ if $n \ge p$. Hence, for $i > 1$, $y_{ij}^{p} \otimes (y'_{ij})^{p} = 0$ and thus,

$$\Delta(H_{1}^{p}) \subseteq H_{1}^{p} \otimes H_{1}^{p}.$$
Thus, $H_{1}^{p}$ is a commutative Hopf algebra in the category of $\k$-vector spaces. Exactly the same argument also shows that $A_{1}^{p}$ is a comodule for $H_{1}^{p}$, where the comodule morphism is obtained by restricting the $H$-comodule morphism associated to $A$. 

\vspace{0.2cm}

Since $F(A)$ is a finitely generated commutative ind-algebra in $\Ver_{p}$, Theorem \ref{Inv} for $\Ver_{p}$ shows that $A_{1} = F(A)^{\inv}$ is a finitely generated commutative ind-algebra in $\Vec \subseteq \Ver_{p}$, and thus so is $A_{1}^{p}$. Additionally, $F(A)^{\inv}$ is integral over $A_{1}^{p}$ and is hence a finite module over $A_{1}^{p}$. Let $A'$ be the invariants of $A_{1}^{p}$ under the coaction of $H_{1}^{p}$. Because the $H_{1}^{p}$-comodule morphism of $A_{1}^{p}$ is defined by restricting the $H$-comodule morphism of $A$ to $A_{1}^{p}$, $A' \subseteq F(A^{\inv})$. Also, as $H_{1}^{p}$ is an finite dimensional commutative Hopf algebra in the category of $\k$-vector spaces, and $A_{1}^{p}$ is a finitely generated commutative algebra in $\Rep(H_{1}^{p})^{\ind}$, a classical theorem of Demazure and Gabriel \cite[Ch. III, \textsection 2, 6.1]{DG} now states that $A_{1}^{p}$ is finitely generated as a module over $A'$, which is a finitely generated algebra in $\Rep(H_{1}^{p})^{\ind}$.

Hence, as $F(A)^{\inv}$ is finite as a module over $A_{1}^{p}$, $F(A)^{\inv}$ is finitely generated as a module over $A'$, a finitely generated $\k$-subalgebra of $F(A^{\inv})$. We now use this to imply Theorem \ref{Inv} for $\C$. The classical Hilbert basis theorem for commutative $\k$-algebras implies that $F(A)^{\inv}$ is a Noetherian module over $A'$. Thus,  $F(A^{\inv})$ is a finitely generated module over $A'$ and is hence a finitely generated commutative $\k$-algebra. Also, as $F(A)^{\inv}$ is finite as a module over $A'$, it is finite over $F(A^{\inv})$. Hence, by Theorem \ref{Inv} for $\Ver_{p}$, $F(A)$ is finite as a module over $F(A)^{\inv}$ and is thus also finite over $F(A^{\inv})$. 

\vspace{0.2cm}

We show that this implies that $A$ is finite as a module over $A^{\inv}$. Let $X$ be a subobject of $F(A)$ in $\Ver_{p}$ such that the canonical multiplication map

$$F(A^{\inv}) \otimes X \rightarrow F(A)$$
is an epimorphism. Since $A^{\inv}$ is an ind-object in $\C$, it is the sum of its subobjects. Hence, there exists some subobject $Y$ of $A^{\inv}$ in $\C$ such that $F(Y)$ contains $X$. Thus, if $\eta$ is the natural multiplication morphism 

$$A^{\inv} \otimes Y \rightarrow A$$
in $\C$, then 

$$F(\eta) : F(A^{\inv} \otimes Y) \cong F(A) \otimes F(Y) \rightarrow F(A)$$
is an epimorphism. This implies that $\coker F(\eta) = 0$, which, by exactness of $F$, implies that $F(\coker \eta) = 0$. As $F$ is faithful, this implies that $\coker \eta = 0$ and hence $\eta$ is an epimorphism. Thus, $A$ is a finite module over $A^{\inv}$. Additionally, since $F(A^{\inv})$ is finitely generated commutative algebra, we can find a subobject $X$ of $F(A^{\inv})$ in $\Ver_{p}$ that generates it as an ind-algebra in $\Ver_{p}$. Choosing a subobject $Y$ of $A^{\inv}$ in $\C$ such that $F(Y)$ contains $X$, a similar argument to the one above shows that $Y$ generates $A^{\inv}$ as an ind-algebra in $\C$. Thus, $A^{\inv}$ is finitely generated, which proves Theorem \ref{Inv} for $\C$.

\vspace{0.5cm}

\section{Proof of Theorem \ref{Inv2}}

Let $\k$ now be an algebraically closed field of characteristic $2$ and consider the category $\sVec_{2}$ defined as in the introduction. To recall, $\sVec_{2}$ was the category of representations of the Hopf algebra $D = \k[d]/(d^{2})$ (with $d$ primitive) equipped with $R$-matrix

$$1 \otimes 1 + d \otimes d.$$ 
Let $\C$ be a symmetric finite tensor category fibered over $\sVec_{2}$. Our goal is to show that finitely generated commutative ind-algebras in $\C$ have finitely generated invariants and are finite over their invariants.

Note that commutative algebras in $\sVec_{2}$ are $\k$-algebras $A$ equipped with a derivation $d$, with $d^{2} = 0$, such that for any $a, b \in A$,

$$[a, b] = d(a)d(b).$$
We call this property $d$-commutativity. When we say that $A$ is commutative, we will mean commutative as a $\k$-algebra. We will specifically use the term $d$-commutativity when we refer to algebras that are commutative in $\sVec_{2}$.

Now, as $\C$ is a symmetric finite tensor category fibered over $\sVec_{2}$, it is the category of comodules of a $d$-commutative Hopf algebra $H$ in $\sVec_{2}$ (Remark \ref{Tannakian} again). Let $A$ be a finitely generated $d$-commutative algebra in $\C$, which we can view as a $d$-commutative $H$-comodule algebra in $\sVec_{2}$. To prove Theorem \ref{Inv2}, we need to show that $A^{\inv}$ is finitely generated and $A$ is a finite $A^{\inv}$-module. 

\vspace{0.2cm}

We first prove that 

$$A^{4} := \{a^{4} : a \in A\}$$
is a commutative $H$-comodule subalgebra of $A$ that is contained in $\ker(d)$. The fact that $A^{4} \subseteq \ker(d)$ is trivial, as $d(a^{4}) = 4a^{3}d(a) = 0.$ In addition, since $A$ is $d$-commutative, this implies that elements in $A^{4}$ commute with each other. Hence, it suffices to prove that $A^{4}$ is a $\k$-algebra and $\Delta(A^{4}) \subseteq A^{4} \otimes H$, where $\Delta : A \rightarrow A \otimes H$ is the comodule structure morphism. In fact, we will show that $\Delta(A^{4}) \subseteq A^{4} \otimes H^{4}$.

We prove that $A^{4}$ is a $\k$-algebra, i.e., that is is closed under sums and products. First, note that for any $a \in A$, 

$$0 = [a, a] = d(a)^{2}.$$
Now, let $a, b \in A$. Then, by $d$-commutativity of $A$,

$$(ab)^{2} = a^{2} b^{2} + ab\,d(a)d(b).$$
Thus, using $d$-commutativity and the fact that $d(a)^{2} = d(b)^{2} = 0$,

$$(ab)^{4} = a^{2}b^{2}a^{2}b^{2} = a^{4}b^{4}.$$
Let $a_{1}, \ldots, a_{n} \in A$. Then, 

$$(a_{1} + \cdots + a_{n})^{2} = \sum_{i} a_{i}^{2} + \sum_{i < j} d(a_{i})d(a_{j}).$$
Every term in the above sum is in the kernel of $d$ (as $d^{2} = 0$) and is hence central in $A$. Thus, again using the fact that $d(a_{i})^{2} = 0$

$$(a_{1} + \cdots + a_{n})^{4} = \sum_{i} a_{i}^{4}.$$
Hence, we see that $A^{4}$ contains the unit and is closed under sums and products. Thus, $A^{4}$ is a $\k$-subalgebra of $A$. 

\vspace{0.2cm}

We next show that $\Delta(A^{4}) \subseteq A^{4} \otimes H^{4}$. Since $A \otimes H$ is also a $d$-commutative algebra, our computation above regarding sums of fourth powers works here as well. Hence, for arbitrary $a \in A$, if

$$\Delta(a) = \sum_{i} a_{i} \otimes h_{i},$$
then

$$\Delta(a^{4}) = (\Delta(a))^{4} = \sum_{i} a_{i}^{4} \otimes h_{i}^{4}.$$
This shows that $\Delta(A^{4}) \subseteq A^{4} \otimes H^{4}$. By the same computation as in the case of $A$, $H^{4}$ is a commutative $\k$-algebra. In fact, since the comultiplication in $H$ is the $H$-comodule map for $H$, using the same computation as for $\Delta$ above, we see that $H^{4}$ is closed under comultiplication. Finally, if $S$ is the antipode, and $h \in H$, then it is clear that

$$S(h^{4}) = S(h)^{4}$$
and hence $H^{4}$ is closed under the antipode as well. Thus, $H^{4}$ is a commutative Hopf algebra over $\k$ and $A^{4}$ is a commutative $H^{4}$-comodule algebra. Hence, since $A^{4}$ is finitely generated (because $A$ is finitely generated), by \cite[Ch. III, \textsection 2, 6.1]{DG}, $(A^{4})^{\inv}$ is finitely generated over $\k$ and $A^{4}$ is finite over $(A^{4})^{\inv}$. 

Thus, since $A$ is finite as a module over $A^{4}$, it is finite as a module over $(A^{4})^{\inv} \subseteq A^{\inv}$. Since $(A^{4})^{\inv}$ is a finitely generated commutative $\k$-algebra, the classical Hilbert basis theorem implies that $A$ is a Noetherian $(A^{4})^{\inv}$-module. Thus, as

$$(A^{4})^{\inv} \subseteq A^{\inv} \subseteq A$$
$A^{\inv}$ is a finitely generated $(A^{4})^{\inv}$-module and hence a finitely generated $\k$-algebra. Additionally, as $A$ is finite over $(A^{4})^{\inv}$, $A$ is finite over $A^{\inv}$. This proves Theorem \ref{Inv2} for $\C$.

\bibliographystyle{alphanum}

\bibliography{Hilbert-Basis}

\end{document}